\documentclass[a4paper,12pt,twoside]{article}
\usepackage{amsmath,amsthm,amscd,amsfonts,amssymb,amsxtra}
\usepackage[hmargin=1.25in,vmargin=1.25in]{geometry}
\usepackage[utf8]{inputenc}
\usepackage{graphicx}
\usepackage{microtype}
\usepackage[auth-sc-lg,affil-it]{authblk}
\usepackage[pdftex,bookmarks,colorlinks=false]{hyperref}
\usepackage{booktabs}
\usepackage{enumitem}
\usepackage{caption,subcaption}
\usepackage[mathscr]{euscript}
\usepackage{setspace}

\newcommand{\sP}{\mathscr{P}}

\newtheorem{thm}{Theorem}[section]

\newtheorem{cor}[thm]{Corollary}
\newtheorem{defn}[thm]{Definition}

\newtheorem{lem}[thm]{Lemma}

\newtheorem{prob}{Problem}
\newtheorem{prop}[thm]{Proposition}
\newtheorem{rem}[thm]{Remark}

\def\ni{\noindent}

\title{\sc On Disjunctive and Conjunctive Set-Labelings of Graphs}

\author{\large Naduvath Sudev}
\affil{\small Department of Mathematics\\ Vidya Academy of Science \& Technology \\ Thrissur - 680501, Kerala, India.\\ E-mail: sudevnk@gmail.com}

\date{}

\begin{document}
\maketitle

\begin{abstract}
Let $X$ be a non-empty set and $\sP(X)$ be its power set. A set-valuation or a set-labeling of a given graph $G$ is an injective function $f:V(G) \to \sP(X)$ such that the induced function $f^{\ast}:E(G) \to \sP(X)$ defined by $f^{\ast} (uv) = f(u)\ast f(v)$, where $\ast$ is a binary operation on sets. A set-indexer of a graph $G$ is an injective set-valued function $f:V(G) \to \sP(X)$ such that the induced function $f^{\ast}:E(G) \to \sP(X)$ is also injective. In this paper, two types of set-labelings, called conjunctive set-labeling and disjunctive set-labeling, of graphs are introduced and some properties and characteristics of these types of set-labelings of graphs are studied.
\end{abstract}

\ni {\small \bf Key Words}: Set-labeling of graphs; disjunctive set-labelings; conjunctive Set-labeling. 

\vspace{0.2cm}

\ni {\small \bf Mathematics Subject Classification}: 05C78.

\section{Introduction to Set-Valuations of Graphs}

For all  terms and definitions, not defined specifically in this paper, we refer to \cite{BM1,JAG,FH,DBW}. Unless mentioned otherwise, all graphs considered here are simple, finite, connected and non-trivial.

The researches on graph labeling problems attracted wide interest among researchers after the introduction of the concept of number valuations of graphs in \cite{AR1}. Motivated from the problems in social interactions in social networks, a set-analogue of number valuations of graphs, called set-valuations of graphs, has been introduced in \cite{BDA1}. Since then, a lot of researches have been taken place on both number valuations and set-valuations of graphs.

Let $X$ be a non-empty set and $\sP(X)$ be its power set. A {\em set-valuation} or {\em set-labeling} of a graph $G$ with respect to $X$ is an injective set-valued function $f:V(G)\to \sP(X)$, whose induced function $f^{\ast}(uv):E(G)\to \sP(X)$ is defined by $f^{\ast}(uv)=f(u)\ast f(v)$,where $\ast$ is a binary operation on sets. A set-labeling $f$ is said to be a {\em set- indexer} of $G$ if the induced function $f^{\ast}$ is also injective. A graph $G$ which admits a set-labeling (or set-indexer) is called an {\em set-labeled graph} (or {\em set-indexed graph}). Further fundamental and consequent studies on set-valuations and set-indexers of graphs, can be seen in \cite{BDA1,BDA2}.

The binary operation of set-labels used in \cite{BDA1}, is the symmetric difference $\oplus$ of sets. Later, the notions of integer additive set-labeling and sumset labeling of graphs have been introduced in \cite{GA,GS1,GS0}, by using sumset operation of sets instead of symmetric difference of sets, and studies on graphs which admit these types of set-labelings are appeared in subsequent literature.

The main objective of this paper is to introduce another two types of set-valuations. In which the binary operations are union and intersection of sets. In these types of set-labelings the set labels of edges are the union or intersection of the set labels of their end vertices. We also study certain structural properties of the graphs which admit these types of set-labelings.

\section{Disjunctive and Conjunctive Set-Labelings of Graphs}

\ni As a special type of set-labeling of graphs, we introduce the following notion.

\begin{defn}\label{Defn-2.1}{\rm 
A {\em disjunctive set-labeling} (DSL) of a graph $G$ with respect to a non-empty ground set $X$ is an injective set-valued function $f:V(G)\to \sP(X)$, whose induced function $f^{\cup}:E(G)\to \sP(X)$ is defined by $f^{\cup}(uv)=f(u)\cup f(v)$. A graph which admits a disjunctive set-labeling is called a {\em disjunctive set-labeled graph} (DSL-graph).}
\end{defn}

\begin{defn}\label{Defn-2.2}{\rm 
A {\em disjunctive set-valuation} or {\em disjunctive set-labeling} $f$ of a graph $G$ is said to be a {\em disjunctive set-indexer} (DSI) of a graph $G$ if the induced function $f^{\cup}$ is also injective.A graph which admits a disjunctive set-indexing is called a {\em disjunctive set-indexed graph} (DSI-graph).} 
\end{defn}

\begin{rem}{\rm 
It can be noted that a set-valuation (with respect to the symmetric difference of sets, as defined in \cite{BDA1}) of a given graph becomes a disjunctive set-valuation if the set-labels of any two adjacent vertices are mutually disjoint.}
\end{rem}  

It is to be noted that the choice of the ground set $X$ is important in defining a disjunctive set-labeling for given graphs. The most important and much interesting question in this context is whether all graphs admit disjunctive set-labelings with respect to the suitable choice of ground set $X$. The following theorem provides an answer this question.

\begin{thm}\label{Thm-2.2a}
Every graph $G$ admits a disjunctive set-indexer.
\end{thm}
\begin{proof}
Let $V=\{v_1,v_2,v_3\ldots,v_n\}$ be the vertex set of $G$. Consider a non-empty set $X=\{a_1,a_2,a_3,\ldots,a_n\}$. Now, define a function $f:V(G)\to \sP(X)$ such that $f(v_i)=\{a_i\}$, where $1\le i\le n$. Clearly, $f$ is injective. Then, the induced edge function $f^{\cup}:E(G)\to \sP(X)-\{\emptyset\}$ is given by $f^{\cup}(v_iv_j)=\{a_i,a_j\}$. For any two edges $e_r=v_iv_j$ and $e_s=v_kv_l$ in $E(G)$, we have  $f^{\cup}(e_r)=\{a_i,a_j\}\ne f^{\cup}(e_s)=\{a_k,a_l\}$. Therefore, $f^{\cup}$ is also injective. Hence, $f$ is a disjunctive set-indexer of $G$.
\end{proof}

\ni Figure \ref{fig:DSLG} illustrates a graph which admits a disjunctive set-labeling.

\begin{figure}[h!]
\centering
\includegraphics[width=0.6\linewidth]{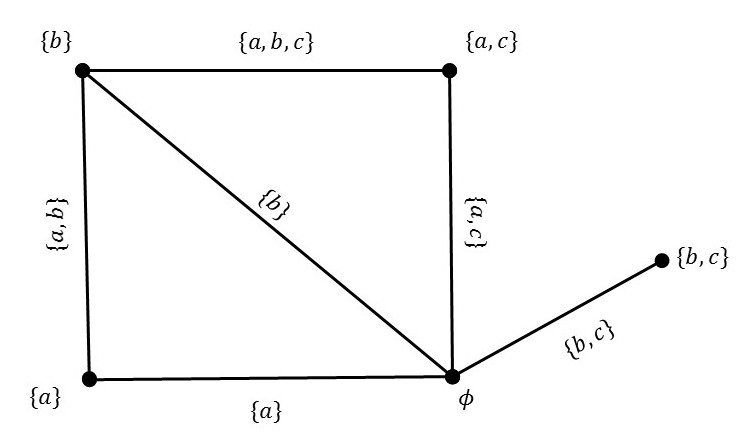}
\caption{An illustration to a DSL-Graph.}
\label{fig:DSLG}
\end{figure}

Analogous to the definition of disjunctive set-labeling of graph, the notion of a conjunctive set-labeling of a graph can be introduce as follows.

\begin{defn}\label{Defn-2.3}{\rm 
A {\em conjunctive set-labeling} (CSL) of a graph $G$ with respect to a non-empty ground set $X$ is an injective set-valued function $f:V(G)\to \sP(X)$, whose induced function $f^{\cap}(uv):E(G)\to \sP(X)$ is defined by $f^{\cap}(uv)=f(u)\cap f(v)$. A graph which admits a conjunctive set-labeling is called a {\em conjunctive set labeled graph} (CSL-graph).}
\end{defn}

\begin{defn}\label{Defn-2.4}{\rm 
A {\em conjunctive set-labeling} $f$ of a graph $G$ is said to be a {\em conjunctive set-indexer} (CSI) of a graph $G$ if the induced function $f^{\cup}$ is also injective.A graph which admits a disjunctive set-indexer is called a {\em disjunctive set-indexed graph} (CSI-graph).} 
\end{defn}

Analogous to Theorem \ref{Thm-2.2a}, the existence of conjunctive set-labeling for any given finite graph
is established in the following theorem. 

\begin{thm}\label{Thm-2.4a}
Every graph $G$ admits a conjunctive set-indexer.
\end{thm}
\begin{proof}
Let $V=\{v_1,v_2,v_3\ldots,v_n\}$ be the vertex set of $G$. Consider a non-empty set $X=\{a_1,a_2,a_3,\ldots,a_n\}$. Now, define a function $f:V(G)\to \sP(X)$ such that $f(v_i)=X-\{a_i\}$, where $1\le i\le n$. Clearly, $f$ is injective. Then, the induced edge function $f^{\cap}:E(G)\to \sP(X)-\{\emptyset\}$ is given by $f^{\cap}(v_iv_j)=X-\{a_i,a_j\}$. For any two edges $e_r=v_iv_j$ and $e_s=v_kv_l$ in $E(G)$, we have  $f^{\cap}(e_r)=X-\{a_i,a_j\}\ne f^{\cup}(e_s)=X-\{a_k,a_l\}$. Therefore, $f^{\cap}$ is also injective. Hence, $f$ is a conjunctive set-indexer of $G$.
\end{proof}

\ni Figure \ref{fig:CSLG} illustrates a graph which admits a conjunctive set-labeling.

\begin{figure}[h!]
\centering
\includegraphics[width=0.6\linewidth]{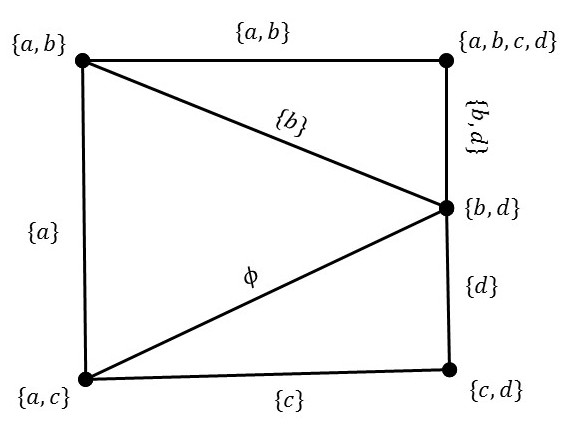}
\caption{An illustration to a CSL-Graph.}
\label{fig:CSLG}
\end{figure}

\begin{rem}\label{Rem-2.5}{\rm 
Since the set-label of every edge of a DSL-graph $G$ is the union of the set-labels of its end vertices, it can be noted that the null set $\emptyset$ will not be the set-label of any edge of $G$. }
\end{rem}

\begin{rem}\label{Rem-2.6}{\rm 
Since the set-label of every edge of a CSL-graph $G$ is the intersection of the set-labels of its end vertices, the ground set $X$ will not be the set-label of any edge of $G$.}
\end{rem}

The structural properties of DSI-graphs and CSI-graphs arise much interests. The following results discuss certain structural properties of DSI-graphs whose vertex set-labels form a topology on the ground set $X$.

\begin{lem}\label{Lem-2.7}
If a connected graph $G$ admits a disjunctive set-indexer, then $f(V(G))=f^{\cup}(E(G))\cup \{\emptyset\}$.
\end{lem}
\begin{proof}
Let $f$ be a DSL defined on a given graph $G$ such that $f(V(G))$ is a topology on the ground set $X$. Now, for every edge $uv\in E(G)$, $f^{\cup}(uv)=f(u)\cup f(v)$ and $f(u),f(v)\in f(V)\implies f(u)\cup f(v)\in f(V)$, since $f(V)$ is a topology on $X$. That is, $f^{\cup}(uv) \in f^{\cup}(E(G))\implies f^{\cup}(uv)\in f(V)$. That is, $f^{\cup}(E(G))\subset f(V)$. Moreover, by Remark \ref{Rem-2.5}, $\emptyset \notin f^{\cup}(E(G))$. Hence, 
\begin{eqnarray}
{\nonumber}|f^{\cup}(E(G))| & \le & |f(V(G))|-1\\ 
\implies |E(G)| & \le & |V(G)|-1. \label{Equn-1}
\end{eqnarray}
Since $G$ is a connected graph, we have 
\begin{equation}
|E(G)|\ge |V(G)|-1  \label{Equn-2}
\end{equation}
From Equation (\ref{Equn-1}) and Equation (\ref{Equn-2}), we have $|E(G)|= |V(G)|-1$. Hence, we have $f(V(G))=f^{\cup}(E(G))\cup \{\emptyset\}$.
\end{proof}

The following result for conjunctive set-labeled graphs can be proved exactly as in the the corresponding result of disjunctive set-labeled graphs.

\begin{lem}\label{Lem-2.7a}
If a connected graph $G$ admits a conjunctive set-indexer, then $f(V(G))=f^{\cap}(E(G))\cup \{X\}$.
\end{lem}

In view of Lemma \ref{Lem-2.7}, a disjunctive set-indexed graph, a necessary and sufficient condition for $f(V)$ to be a topology on the ground set $X$.  

\begin{thm}\label{Thm-2.8}
For a connected disjunctive set-indexed graph $G(V,E)$, the collection $f(V(G))$ is a topology on the ground set $X$ if and only if $G$ is a tree.   
\end{thm}
\begin{proof}
Let $f$ be a DSI defined on a graph $G$ so that $f(V(G))$ is a topology on the ground set $X$. Then, by Lemma \ref{Lem-2.7}, we have $|f(V)|=|f^{\cup}(E)|+1$. That is, $|V|=|E|+1$. Therefore, $G$ is a tree. 

Conversely, assume that $G$ be a tree on $n$ vertices, say $V(G)=\{v_1,v_2,v_3,\ldots v_n\}$. What required is to define a DSL $f$ on $G$ such that $f(V(G))$ is a topology on $X$. For this, let $X=\{a_1,a_2,a_3,\ldots, a_{n-1}\}$. Now, define the function $f:V(G)\to \sP(X)$ such that $f(v_1)=\emptyset, f(v_2)=\{a_1\}, f(v_3)=\{a_1,a_2\}, \ldots, f(v_i)=\{a_1,a_2,a_3,\ldots, a_{i-1}\},\dots, f(v_n)=X$. Then, we have $f(V)=\{\emptyset,\{a_1\},\{a_1,a_2\},\ldots,\\ \{a_1,a_2,a_3,\ldots, a_{i-1}\}, X\}$, which is clearly a topology on $X$. This completes the proof.
\end{proof}

Figure \ref{fig:TDSLG} is an illustration to a DSI-graph, the collection of whose vertex set-labels is a topology on the ground set $X$.

\begin{figure}[h!]
\centering
\includegraphics[width=0.5\linewidth]{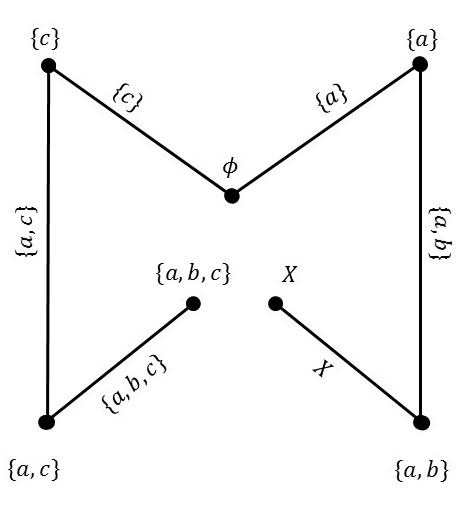}
\caption{An illustration to a CSL-Graph.}
\label{fig:TDSLG}
\end{figure}

Using the similar arguments given in the above theorem, the corresponding result for conjunctive set-labeled graphs can be established as follows.

\begin{thm}\label{Thm-2.8a}
For a connected CSI-graph $G(V,E)$, the collection $f(V(G))$ is a topology on the ground set $X$ if and only if $G$ is a tree.   
\end{thm}

Another important property of a set-labeling is its gracefulness. This property of disjunctive set-labeling of graphs can be defined as follows.

\begin{defn}{\rm 
A disjunctive set-indexer $f$ of a graph $G$ is said to be a {\em graceful disjunctive set-labeling} (graceful DSI, in short) if $f^{\cup}(E(G))=\sP(X)-\{\emptyset\}$.}
\end{defn}

Similarly, we have the following definition of graceful conjunctive set-indexer of graphs.

\begin{defn}{\rm 
A conjunctive set-indexer $f$ of a graph $G$ is said to be a {\em graceful conjunctive set-labeling} (graceful CSI, in short) if $f^{\cup}(E(G))=\sP(X)-\{\emptyset\}$.}
\end{defn}

The following result discusses a necessary and sufficient condition for a DSI of $G$ to be a graceful DSL. 

\begin{prop}\label{Prop-2.9}
A disjunctive set-indexer of a given graph $G$ is graceful if and only if $f(V)$ is the discrete topology on the ground set $X$.
\end{prop}
\begin{proof}
The proof is an immediate consequence of the fact that $|f(V)|=|f^{\cup}(E)|+1$, if $f(V)$ is a topology on the ground set $X$.
\end{proof}

\begin{cor}\label{Cor-2.10}
A graph $G$ which admits a graceful disjunctive set-indexer has even number of vertices and odd number of edges.
\end{cor}
\begin{proof}
If $f$ is a graceful DSL, then by Proposition \ref{Prop-2.9}, $f(V)$ is the indiscrete topology on the ground set $X$. Then, by Theorem \ref{Thm-2.8}, $G$ is a tree.  Therefore, $|E|=|V|-1$.

Moreover, since $f$ is graceful, we have $|f(V)|=|f^{\cup}(E)|+1=2^{|X|}$. Therefore, $|f(V)|=|V|$ is even and $|f^{\cup}(E)|=|E|$ is odd.
\end{proof}

The corresponding results on conjunctive set-indexers of graphs can be stated as follows.

\begin{prop}\label{Prop-2.9a}
A conjunctive set-indexer of a given graph $G$ is graceful if and only if $f(V)$ is the discrete topology on the ground set $X$.
\end{prop}

\begin{cor}\label{Cor-2.10a}
A graph $G$ which admits a graceful disjunctive set-indexer has even number of vertices and odd number of edges.
\end{cor}

The proofs of above two results follow in an exactly similar way as that of the corresponding results on DSI-graphs. 

The choice of the ground set plays an important role in defining a DSL and a CSL for given graph classes. Hence, we introduce the following notions.

\begin{defn}{\rm 
The minimum cardinality of the ground set $X$ required for a given graph $G$ to admit a DSI is called the \textit{disjunctive set-indexing number} of that graph and is denoted by $\varrho(G)$.}
\end{defn}

Similarly, we can define the notion of the conjunctive set-indexing number of given graphs as follows. 

\begin{defn}{\rm 
The minimum cardinality of the ground set $X$ required for a given graph $G$ to admit a CSI is called the \textit{conjunctive set-indexing number} of that graph and is denoted by $\varpi(G)$.}
\end{defn}  

\ni The following theorem determines certain bounds for the disjunctive set-indexing number of a graph.

\begin{thm}
The disjunctive set-indexing number of a graph $G$ on $n$ vertices is  $\varrho=\lceil \log_2n\rceil$.
\end{thm}
\begin{proof}
Let $G$ be a graph on $n$ vertices which admits a disjunctive set-indexer $f$ with respect to the ground set $X$. Then, we have 
\begin{eqnarray*}
& 2^{|X|-1}< n\le 2^{|X|}\\
\implies & |X|-1 < \lfloor \log_2 n \rfloor \le |X|\\
\implies & \lfloor \log_2 n \rfloor \le |X| \le 1+\lfloor \log_2 n \rfloor\\
\implies & |X|\le \lceil \log_2 n \rceil.
\end{eqnarray*} 
This completes the proof.
\end{proof}

In a similar way, we can also determine the conjunctive set-indexing number of a graph as follows.

\begin{thm}
The conjunctive set-indexing number of a graph $G$ on $n$ vertices is  $\varpi=\lceil \log_2n\rceil$.
\end{thm}
\begin{proof}
The proof follows exactly as in the previous theorem.
\end{proof}

\section{Scope for Further Studies}

In this paper, the notions of two types of set-labeling of graphs have been introduced and certain properties and characteristics of graphs which admit these types of set-labels have been discussed. More problems in this area are still open. Some of the open problems we have identified in this area are the following.


A disjunctive (or conjunctive) set-indexer $f$ of a graph $G$ is called \textit{sequential} if $f(V)\cup f^{\cup}(E)=\sP(X)$ and a conjunctive set-indexer $f$ of a graph $G$ is called \textit{sequential} if $f(V)\cup f^{\cap}(E)=\sP(X)$. 

\begin{prob}{\rm
Characterise the graphs which admit sequential disjunctive set-indexer (or sequential conjunctive set-indexer).}
\end{prob}


A disjunctive set-indexer $f$ of a graph $G$ is called \textit{topogenic} if $f(V)\cup f^{\cup}(E)$ is a topology on the ground set $X$ and a conjunctive set-indexer $f$ of a graph $G$ is called \textit{topogenic} if $f(V)\cup f^{\cap}(E)$ is a topology on the ground set $X$

\begin{prob}{\rm
Characterise the graphs which admit topogenic disjunctive set-indexer (or topogenic conjunctive set-indexer).}
\end{prob}

A set-labeling of a graph $G$ is called a \textit{set-magic labeling} (\cite{BDA1},\cite{BDA2}) if the set-label of a vertex is the union of the set-labels of its end vertices. Then, 

\begin{prob}{\rm 
Determine the conditions required for a conjunctive set-labeling of a graph to be a magic set-labeling of $G$. }
\end{prob}
Determining the disjunctive and conjunctive set-indexing numbers of various graph classes, graph operations, graph products and graph powers seems to be much promising for further investigation.

More properties and characteristics of various set-labeled graphs are yet to be investigated. The problems of establishing the necessary and sufficient conditions for various graphs and graph classes to have certain other types of set valuations are also open. All these facts highlight a wide scope for future studies in this area.

\section*{Acknowledgement}

The author of this article dedicates this paper to the memory Prof. (Dr.) D. Balakrishnan, Founder Academic Director, Vidya Academy of Science and Technology, Thrissur, India., who had been the mentor, the philosopher and the real role model for many teachers, including the author, in both teaching and research.

\end{document}